\documentclass[12pt,tbtags,leqno]{amsart}

\usepackage{amssymb}
\usepackage{amsthm}
\usepackage{amsmath}
\usepackage{verbatim}
\usepackage[numbers,sort&compress]{natbib}

\numberwithin{equation}{section}

\newcommand{\HH}{\mathcal{H}}
\newcommand{\HK}{\mathcal{K}}
\newcommand{\HM}{\mathcal{M}}
\newcommand{\HE}{\mathcal{E}}
\newcommand{\HF}{\mathcal{F}}

\newcommand{\HD}{\mathcal{D}}

\newcommand{\HA}{\mathcal{A}}

\newcommand{\D}{\mathbb{D}}

\theoremstyle{plain}
\newtheorem{theorem}{Theorem}[section]
\newtheorem{lemma}[theorem]{Lemma}

\newtheorem{conjecture}[theorem]{Conjecture}
\newtheorem{prop}[theorem]{Proposition}

\theoremstyle{definition}

\makeatletter
\@namedef{subjclassname@2020}{%
  \textup{2020} Mathematics Subject Classification}
\makeatother

\begin{document}
\title[Vector-Valued Sub-Bergman Spaces]{Vector-Valued Sub-Bergman Spaces: Range Inclusions and Nonrigidity}

\author{Shuaibing Luo}
\address{School of Mathematics, Hunan University, Changsha, Hunan, 410082, PR China}
\email{sluo@hnu.edu.cn}

\author{Jiming Shen}
\address{School of Mathematics, Hunan University, Changsha, Hunan, 410082, PR China}
\email{jimishen@hnu.edu.cn}
\date{}

\maketitle

\begin{abstract}
Let $\mathcal E$ be a separable Hilbert space,
$B\in H^\infty_1(L(\mathcal E))$, and $\alpha>-1$.  Associated
with the Toeplitz operator $T_B$ on the vector-valued weighted
Bergman space $A^2_{\alpha,\mathcal E}$ are the sub-Bergman spaces $\mathcal A_\alpha(B)$ and $\mathcal A_\alpha(B^*)$.
In this paper, we prove that
\[
 \mathcal A_\alpha(B)\supset
 A^2_{\alpha-1,\mathcal H(B(0))}
 \quad\text{and}\quad
 \mathcal A_\alpha(B^*)\supset
 A^2_{\alpha-1,\mathcal H(B(0)^*)}.
\]
This confirms a conjecture of \cite{GLM26}.
We also construct an infinite-dimensional
two-sided inner function $B$ such that $\|B(0)\|<1$ and
\[
 \mathcal A_\alpha(B)\approx
 \mathcal A_\alpha(B^*)\approx
 A^2_{\alpha-1,\mathcal E},
\]
but $B$ is not a finite Blaschke-Potapov product. This example shows that the finite-dimensional rigidity theorem of \cite{GLM26} does not extend to infinite-dimensional coefficient spaces.
\end{abstract}

\section{Introduction}
Let $\HH$ and $\HK$ be two separable Hilbert spaces.
For a contraction $A$ from $\HH$ to $\HK$, let $\HM(A)$ be the operator range of $A$ with the Hilbert space structure determined by
\[
\left\langle Ah_{1},Ah_{2}\right\rangle _{\mathcal{M}(A)}=\left\langle
h_{1},h_{2}\right\rangle _{\HH}, \quad h_{1}\in(\ker A)^{\perp}, h_{2}\in \HH.
\]
Equivalently, $A$ is a coisometry from $\HH$ onto $\HM(A)$.
The de Branges-Rovnyak space $\HH(A)$ is the space $\HM(D_{A^*})$ , where $D_{A^*} = (I-AA^*)^{1/2}$.
Let $H^\infty$ be the space of all bounded analytic functions on the unit disk $\D$, and $H^\infty_1$ be the unit ball of $H^\infty$.
When $\HH = \HK$ is the Hardy space $H^2$ on the unit disk, $A = T_b^{H^2}$ is the Toeplitz operator on $H^2$ with $b \in H^\infty_1$, the
de Branges-Rovnyak spaces $\HH(A)$ and $\HH(A^*)$ are also called sub-Hardy spaces (\cite{debranges, sarasonbook,hbspaces1fricainmashreghi}).
De Branges-Rovnyak spaces have a rich structure and play an important role in complex analysis and operator theory. In particular,
they are fundamental in the model theory of Hilbert space contractions (\cite{BallBoloBasics}) and analytic expansive operators (\cite{AM, LGR}).

Zhu introduced the Bergman-space analogues of these spaces in
\cite{Zhu96, Zhu03}. If $b \in H^\infty_1$, the defect operators of multiplication by $b$ and its
adjoint on a Bergman space give two sub-Bergman spaces, $\HA(b)$ and $\HA(\overline{b})$.
Their structure is subtler than that of the corresponding sub-Hardy
spaces, in part because invariant subspaces of the Bergman shift are
considerably more complicated. This construction extends naturally to de Branges-Rovnyak
spaces contractively contained in the weighted Bergman space $A_{\alpha}^{2}, \alpha>-1$.
If $A = T_b^{A_{\alpha}^{2}}$, the resulting spaces $\HH(A)$ and $\HH(A^*)$ are denoted by $\HA_\alpha(b)$ and $\HA_\alpha(\overline{b})$, respectively.
When $\alpha = 0$, we have $\HA_0(b) = \HA(b)$, $\HA_0(\overline{b}) = \HA(\overline{b})$.
Since the work of Zhu \cite{Zhu96, Zhu03}, there have been
several papers answering Zhu's questions and extending his results, see e.g.
\cite{Abkar, AD, chu1, chu2, LZ, Nowak, RS20, Sultanic, Symesak}.

The study of vector-valued sub-Bergman spaces was initiated in \cite{GLM26}. For $\alpha > -1$ and a
separable Hilbert space $\mathcal E$, let $A^2_{\alpha,\mathcal E}$ denote
the $\mathcal E$-valued reproducing kernel Hilbert space with kernel
\[
 K_\alpha(z,w)
 =\frac{I_{\mathcal E}}{(1-z\overline w)^{\alpha+2}}.
\]
Let $H^\infty_1(L(\mathcal E))$ denote the unit ball of
$H^\infty(L(\mathcal E))$.
For
$B\in H^\infty_1(L(\mathcal E))$, multiplication
by $B$ is a contraction on $A^2_{\alpha,\mathcal E}$, and we define
\[
 \mathcal A_\alpha(B)
 =\mathcal M\bigl((I-T_BT_B^*)^{1/2}\bigr),
 \qquad
 \mathcal A_\alpha(B^*)
 =\mathcal M\bigl((I-T_B^*T_B)^{1/2}\bigr).
\]
In contrast with the scalar setting, these two spaces must be treated
separately: the noncommutativity of operator-valued symbols can make
their behavior quite different.  Infinite-dimensional coefficient
spaces introduce a further difficulty, since the ranges of the
pointwise defect operators need not be closed.  Moreover, pure
contractivity of $B$ need not imply $\|B(0)\|<1$.

In \cite{GLM26}, the authors give a criterion
for both sub-Bergman spaces to contain the full space
$A^2_{\alpha-1,\mathcal E}$ and precise description when $B$ is a
finite Blaschke-Potapov product. More specifically, if
$B(z)=\prod_{j=1}^n Q(z,a_j,\mathcal F_j,\mathcal E)$ is a
finite Blaschke-Potapov product (see Section 2 for the definition), then
\[
 \mathcal A_\alpha(B)
 \approx A^2_{\alpha-1,\mathcal H(B(0))},
 \quad
 \mathcal A_\alpha(B^*)
 \approx A^2_{\alpha-1,\mathcal H(B(0)^*)},
\]
where $\approx$ denotes equality as sets.  These results lead the authors
of \cite{GLM26} to conjecture that, when $\mathcal E$ is infinite
dimensional, every operator-valued Schur function $B$ satisfies
\[
 \mathcal A_\alpha(B)\supset
 A^2_{\alpha-1,\mathcal H(B(0))},
 \quad
 \mathcal A_\alpha(B^*)\supset
 A^2_{\alpha-1,\mathcal H(B(0)^*)}.
\]
See \cite[Conjecture 4.18]{GLM26}.
The goal of this paper is to prove this conjecture
without imposing any finite-dimensional assumption on $\mathcal E$ and
without assuming that $B$ is inner or rational. Our main result is the following.
\begin{theorem}
\label{mainthm1}
Let $\HE$ be a separable Hilbert space.
Let $B\in H^{\infty}_1(L(\HE))$. Then for $\alpha>-1,$
$$\mathcal{A}_{\alpha}(B)\supseteq
A_{\alpha-1,\HE_{1}}^{2}, \quad \mathcal{A}%
_{\alpha}(B^{\ast})\supseteq A_{\alpha-1,\HE_{2}}^{2},$$
where $\HE_{1}=\mathcal{H}(B(0))$ and $\HE_{2}%
=\mathcal{H}(B(0)^{\ast})$.
\end{theorem}

The key ingredient to prove the above theorem is a Harnack-type estimate which retains the pointwise
defect at the origin.  We show that for every $z\in\mathbb D$,
\[
\begin{aligned}
 I_{\mathcal E}-B(z)^*B(z)
 &\geq \frac{1-|z|^2}{4}
       \bigl(I_{\mathcal E}-B(0)^*B(0)\bigr),\\
 I_{\mathcal E}-B(z)B(z)^*
 &\geq \frac{1-|z|^2}{4}
       \bigl(I_{\mathcal E}-B(0)B(0)^*\bigr).
\end{aligned}
\]

Theorem \ref{mainthm1} also gives short proofs of two structural results from
\cite{GLM26}. First,
\[
 \mathcal A_\alpha(B)\supset A^2_{\alpha-1,\mathcal E}
 \quad\Longleftrightarrow\quad
 \|B(0)\|<1
 \quad\Longleftrightarrow\quad
 \mathcal A_\alpha(B^*)\supset A^2_{\alpha-1,\mathcal E}.
\]
Second, $\mathcal A_\alpha(B)$ and $\mathcal A_\alpha(B^*)$ are dense
in $A^2_{\alpha,\mathcal E}$ if and only if $B$ is purely contractive.

Finally, we show that a rigidity phenomenon established in
\cite[Theorem 1.4]{GLM26} is genuinely finite dimensional.
\begin{theorem}\label{mainthm2}
There exist an
infinite dimensional Hilbert space $\HE$ and a function $B\in H^{\infty}_1(L(\HE))$
such that for every $\alpha>-1$,
\[
 \HA_\alpha(B)
 \approx
 \HA_\alpha(B^*)
 \approx
 A^2_{\alpha-1,\HE},
\]
while $B$ is not a finite Blaschke-Potapov product.
\end{theorem}
In \cite[Theorem 1.4]{GLM26}, it is proved that when $\dim \HE < \infty$, then
\[
 \HA_\alpha(B)
 \approx
 \HA_\alpha(B^*)
 \approx
 A^2_{\alpha-1,\HE},
\]
if and only if $B$ is a finite Blaschke-Potapov product and pure. The above result shows that
the finite-dimensional characterization in
\cite[Theorem 1.4]{GLM26} cannot be extended to
infinite-dimensional coefficient spaces.

The paper is organized as follows.  In Section~2, we first review operator ranges,
the relevant reproducing kernels, and the pure/unitary decomposition of
an operator-valued Schur function. Then we prove a Harnack-type inequality and Theorem \ref{mainthm1},
 and derive several consequences.
In Section~3, we construct the infinite-dimensional example that proves Theorem \ref{mainthm2}.

\section{Characterization of $\HA_\alpha(B)$ and $\HA_\alpha(B^*)$}
Let $\HD$ and $\HE$ be separable Hilbert spaces and $L(\HD, \HE)$ the space of bounded linear operators from $\HD$ to $\HE$. Let $H^\infty(L(\HD, \HE))$ be the $L(\HD, \HE)$-valued $H^\infty$ functions and $H^\infty_1(L(\HD, \HE))$ the unit ball of $H^\infty(L(\HD, \HE))$. If $B \in H^\infty_1(L(\HD, \HE))$, then we say that $B$ is a contractive operator-valued analytic function. For $\alpha > -1$, let $A^2_{\alpha, \HE}$ be the $\HE$-valued Bergman space on the unit disk $\D$ whose reproducing kernel is given by
\begin{align*}
K_\alpha(z,w) = \frac{I_\HE}{(1-z\overline{w})^{\alpha+2}}, \quad z, w \in \D.
\end{align*}
Let $B \in H^\infty_1(L(\HD, \HE))$. Then the Toeplitz operator $T_B$ is a contraction from $A^2_{\alpha, \HD}$ to $A^2_{\alpha, \HE}$. We thus have the de Branges-Rovnyak spaces $\HH(T_B) = \HM(D_{T_B^*})$ and $\HH(T_B^*) = \HM(D_{T_B})$, which we denote by $\HA_\alpha(B)$ and $\HA_\alpha(B^*)$, respectively. The reproducing kernels of $\HA_\alpha(B)$ and $\HA_\alpha(B^*)$ are
\begin{align}\label{reprokernelaabbs}
&F_\alpha^B(z,w) = \frac{I_\HE - B(z) B(w)^*}{(1-z\overline{w})^{\alpha+2}},\\
&F_\alpha^{B^*}(z,w) =\int_{\mathbb{D}}\frac{I_\HD-B^{\ast}(u)B(u)}%
{(1-z\overline{u})^{\alpha+2}(1-u\overline{w})^{\alpha+2}}dA_{\alpha}(u), \quad z, w \in \D, \notag
\end{align}
respectively, see \cite[Theorem 3.1, Proposition 3.2]{GLM26}.

Let $B \in H^\infty_1(L(\HD, \HE))$. The function $B$ is called purely contractive (or just pure) if $\|B(0) v\| < \|v\|$ for all $v \in \HD\backslash \{0\}$. Any contractive $B \in H^\infty_1(L(\HD, \HE))$ admits a decomposition such that
$$B(z) = B_1(z) \oplus U: \HD = \HD_1 \oplus \HD_2 \rightarrow \HE = \HE_1 \oplus \HE_2,$$
where $B_1 \in H^\infty_1(L(\HD_1, \HE_1))$ is pure and $U$ is a constant unitary from $\HD_2$ onto $\HE_2$. In fact,
\begin{align}\label{unitarypart}
\HD_2 = \{v \in \HD: v = B(0)^*B(0)v\}, \quad \HE_2 = \{v \in \HE: v = B(0)B(0)^*v\}.
\end{align}
The $B_1$ is referred to be the purely contractive part and $U$ is the unitary part of $B$, see \cite[Page 194]{SFBK10}. Thus with respect to the decomposition $A^2_{\alpha,\HE} = A^2_{\alpha,\HE_1} \oplus A^2_{\alpha,\HE_2}$, we have
\begin{align}\label{decomdefect}
I - T_B T_B^* = (I - T_{B_1}T_{B_1}^*) \oplus 0.
\end{align}
So $\HA_\alpha(B) = \HA_\alpha(B_1) \oplus \{0\} \subseteq A^2_{\alpha,\HE_1} \oplus A^2_{\alpha,\HE_2}$. Similarly, $\HA_\alpha(B^*) = \HA_\alpha(B_1^*) \oplus \{0\} \subseteq A^2_{\alpha,\HD_1} \oplus A^2_{\alpha,\HD_2}$.
\begin{lemma}\label{pureope}
Let $B \in H^\infty_1(L(\HD, \HE))$. Then
\begin{itemize}
\item[(i)] $\HH(B(0))$ is dense in $\HE$ if and only if $B$ is pure; this is also equivalent to the statement that $\HH(B(0)^*)$ is dense in $\HD$.
\item[(ii)] $\HH(B(0)) = \HE$ if and only if $\|B(0)\| <1$; this is also equivalent to $\HH(B(0)^*) = \HD$.
\end{itemize}
\end{lemma}
\begin{proof}
(i) Note that $\HH(B(0))$ is the range of the positive operator $(I_\HE - B(0)B(0)^*)^{1/2}$ on $\HE$. If $B$ is not pure, then it is clear that $\HH(B(0))$ is not dense in $\HE$. If $B$ is pure, we claim that $I_\HE - B(0)B(0)^*$ is one-to-one. Note that $v \in \ker(I_\HE - B(0)B(0)^*)$ if and only if $v \in \HE_2$, where $\HE_2$ is defined by (\ref{unitarypart}). Let $v \in \HE_2$. If $v \neq 0$, then $B(0)^* v \neq 0$. Thus $B(0)^* B(0) B(0)^* v = B(0)^* v$, and so $B(0)^* v \in \HD_2$. But since $B$ is pure, we must have $B(0)^* v =0$. So $\HE_2 = 0$ and $I_\HE - B(0)B(0)^*$ is one-to-one. Thus $\HH(B(0))$ is dense in $\HE$. Similarly, $\HH(B(0)^*)$ is dense in $\HD$ if and only if $B$ is pure.

(ii) If $\|B(0)\| <1$, then $I_\HE - B(0)B(0)^*$ is invertible and $\HH(B(0)) = \HE$. If $\HH(B(0)) = \HE$, then the positive operator $(I_\HE - B(0)B(0)^*)^{1/2}$ is bounded below and $\|B(0)\| <1$. Similarly, $\|B(0)\| <1$ is equivalent to $\HH(B(0)^*) = \HD$.
\end{proof}

\subsection{Finite Blaschke-Potapov product}

Let $\HE$ be a separable Hilbert space.
Recall that $B \in H^\infty_1(L(\HE))$ is (left) inner if $B(z)^*B(z) = I_\HE$ for a.e. $|z| = 1$. Let $\HF \subseteq \HE$ be a closed subspace of $\HE$ and $P_\HF$ the orthogonal projection from $\HE$ onto $\HF$. The inner function
$$Q(z,a,\HF,\HE) = \varphi_a(z) (I_\HE - P_\HF) + P_\HF$$
is called a Blaschke-Potapov factor, where $\varphi_a(z) = \frac{z-a}{1-\overline{a}z}, a \in \D$. In our discussion, when we say $Q(z,a,\HF,\HE)$ is a Blaschke-Potapov factor we always assume that $\HF \neq \HE$. $B$ is called a finite Blaschke-Potapov product if there are closed subspaces $\HF_i \subseteq \HE$, $a_i \in \D, i = 1, \ldots, n$ and a unitary operator $U \in L(\HE)$ such that
\begin{align}\label{finiteblapot}
B(z) = U \prod_{i=1}^n Q(z,a_i,\HF_i,\HE).
\end{align}
Note that $\HA_\alpha(B) = U \HA_\alpha(U^*B)$, we will assume $U = I_\HE$ in the following. Let $B$ be a finite Blaschke-Potapov product defined by (\ref{finiteblapot}) with $U = I_\HE$. Then by \cite[Lemma 4.4]{GLM26},
\begin{align}\label{finitedecomp}
B(z) = B_1(z) \oplus I_{\HF'}: \HE = \HE' \oplus \HF' \rightarrow \HE' \oplus \HF',
\end{align}
where $B_1(z) = \prod_{i=1}^n Q(z,a_i,\HF_i',\HE')$ is the purely contractive part and the identity $I_{\HF'}$ is the unitary part, $\HF' = \cap_{i=1}^n \HF_i, \HE' = \HE \ominus \HF'$, $\HF_i' = \HF_i \ominus \HF', i = 1,\ldots, n$.
By the decomposition (\ref{finitedecomp}) and Lemma \ref{pureope}, we have $\HH(B(0))$ and $\HH(B(0)^*)$ are dense subspaces of $\HE'$. Thus when $\HE$ is finite dimensional, then
\begin{align*}
\HH(B(0)) = \HH(B(0)^*) = \HE'  = \HF_1^\perp + \ldots + \HF_n^\perp.
\end{align*}
In fact, when $\HE$ is a general Hilbert space and $B$ is a finite Blaschke-Potapov product, it always holds that
\begin{align*}
\HH(B(0)) = \HH(B(0)^*) = \HF_1^\perp + \ldots + \HF_n^\perp,
\end{align*}
see \cite[Lemma 4.8]{GLM26}.

It was proved in \cite[Theorems 4.12 and 4.13]{GLM26} that if $B(z) = \prod_{i=1}^n Q(z,a_i,\HF_i,\HE)$ is a finite Blaschke-Potapov product, then for $\alpha > -1$,
\begin{align*}
\mathcal{A}_{\alpha}(B)\approx A_{\alpha-1,\HE_{1}}^{2}, \quad \mathcal{A}_{\alpha}(B^{\ast})\approx A_{\alpha-1,\HE_{2}}^{2},
\end{align*}
where $X \approx Y$ means they equal as a set,
$$\HE_1 = \HH(B(0)) = \HF_1^\perp + \ldots + \HF_n^\perp = \HH(B(0)^*) = \HE_2.$$
Consequently, if $B(z) = \prod_{i=1}^n Q(z,a_i,\HF_i,\HE)$ is a finite Blaschke-Potapov product, then for $\alpha>-1,$ $\mathcal{A}_{\alpha}(B)\supset
A_{\alpha-1,\HE_{1}}^{2}$ and $\mathcal{A}_{\alpha}(B^{\ast})\supset
A_{\alpha-1,\HE_{1}}^{2}$ for some dense linear manifold $\HE_{1}$ of $\HE$ if and only
if $B$ is pure, if and only if $\cap_{i=1}^{n}F_{i}=\{0\}.$ See \cite[Corollary 4.17]{GLM26}. This leads the authors of \cite{GLM26} to formualte the following conjecture.
\begin{conjecture}[\cite{GLM26}, Conjecture 4.18]
\label{conjecGLM}
Let $B\in H^{\infty}_1(L(\HE))$. If $\HE$ is infinite
dimensional, then for\ each $\alpha>-1,$ $\mathcal{A}_{\alpha}(B)\supseteq
A_{\alpha-1,\HE_{1}}^{2}$ where $\HE_{1}=\mathcal{H}(B(0))$ and $\mathcal{A}%
_{\alpha}(B^{\ast})\supseteq A_{\alpha-1,\HE_{2}}^{2}$ where $\HE_{2}%
=\mathcal{H}(B(0)^{\ast})$.
\end{conjecture}

\subsection{General contractive function}
In this subsection, we will prove Conjecture \ref{conjecGLM}. We need the following Harnack-type inequality.

\begin{lemma}
\label{harnackineq}
Let $B \in H^\infty_1(L(\HE))$. Then for every $z \in \D$,
\begin{align}\label{ineqharnt}
I_{\mathcal{E}} - B(z)^* B(z) &\geq \frac{1 - |z|^2}{4} \left( I_{\mathcal{E}} - B(0)^* B(0) \right), \\
I_{\mathcal{E}} - B(z) B(z)^* &\geq \frac{1 - |z|^2}{4} \left( I_{\mathcal{E}} - B(0) B(0)^* \right). \notag
\end{align}
\end{lemma}

\begin{proof}
Fix \(z \in \mathbb{D}\) and a unit vector \(x \in \mathcal{E}\). Choose a norm-one functional \(\ell \in \mathcal{E}^*\) such that
\[
|\ell(B(z)x)| = \|B(z)x\|.
\]
Let
\[
h(w) = \ell(B(w)x), \quad w \in \mathbb{D}.
\]
Then \(h \in H_1^\infty\). Set \(a = h(0)\). If \(|a| = 1\), then
\[
1 - |h(z)|^2 \geq 0 = \frac{1 - |z|}{1 + |z|}(1 - |a|^2).
\]
If $a\in\mathbb{D}$, let $\omega=\frac{h(z)-a}{1-\overline{a}h(z)}$, then by Schwarz lemma, $|\omega|\leq |z|$.
Since $h(z)=\frac{a+\omega}{1+\overline{a}\omega}$, we obtain
\[
\begin{aligned}
1-|h(z)|^{2} &= \frac{(1-|a|^{2})(1-|\omega|^{2})}{|1+\overline{a}\omega|^{2}} \\
&\geq \frac{(1-|a|^{2})(1-|z|^{2})}{(1+|a||z|)^{2}} \\
&\geq \frac{1-|z|}{1+|z|}\bigl(1-|a|^{2}\bigr).
\end{aligned}
\]
Moreover, $|a|=|\ell(B(0)x)|\leq \|B(0)x\|$. Hence
\[
\begin{aligned}
\bigl\langle (I_{\mathcal{E}}-B(z)^{*}B(z))x,x\bigr\rangle
&= 1-\|B(z)x\|^{2} \\
&= 1-|h(z)|^{2} \\
&\geq \frac{1-|z|}{1+|z|}\bigl(1-\|B(0)x\|^{2}\bigr) \\
&= \frac{1-|z|}{1+|z|}\bigl\langle (I_{\mathcal{E}}-B(0)^{*}B(0))x,x\bigr\rangle.
\end{aligned}
\]
Since this holds for every unit vector $x$, we obtain
\[
I_{\mathcal{E}}-B(z)^{*}B(z) \geq \frac{1-|z|}{1+|z|}\bigl(I_{\mathcal{E}}-B(0)^{*}B(0)\bigr).
\]
Note that
\[
\frac{1-|z|}{1+|z|}\geq \frac{1-|z|^{2}}{4},
\]
Thus the first inequality in (\ref{ineqharnt}) holds.

Now let $C(w) = B(\overline{w})^*, w \in \D$. Then $C \in H^\infty_1(L(\HE))$. So for every $z \in \D$,
\begin{align*}
I_\HE - B(\overline{z}) B(\overline{z})^* = I_\HE-C(z)^*C(z)
 \geq
 \frac{1-|z|^2}{4}
 \bigl(I_\HE-B(0)B(0)^*\bigr).
\end{align*}
Replacing $z$ by $\overline{z}$, we obtain the second inequality in (\ref{ineqharnt}).
\end{proof}

Now we can prove Conjecture \ref{conjecGLM}. We restate Theorem \ref{mainthm1} in the following.

\begin{theorem}
\label{conjecproof}
Let $\HE$ be a separable Hilbert space. Let $B \in H^\infty_1(L(\HE))$. Then for $\alpha > -1$,
\begin{align*}
\HA_\alpha(B) \supseteq A^2_{\alpha-1,\HE_1}, \quad \HA_\alpha(B^*) \supseteq A^2_{\alpha-1,\HE_2},
\end{align*}
where $\HE_1 = \HH(B(0))$, $\HE_2 = \HH(B(0)^*)$.
\end{theorem}

\begin{proof}
Let
\[
A = B(0), \qquad D_{A^*} = (I_{\mathcal{E}} - AA^*)^{1/2}, \qquad D_A = (I_{\mathcal{E}} - A^*A)^{1/2}.
\]
Then
\[
\mathcal{E}_1 = \mathcal{M}(D_{A^*}), \qquad \mathcal{E}_2 = \mathcal{M}(D_A).
\]
Since \(B\) is analytic, we have \(I - T_B^* T_B = T_{I_{\mathcal{E}} - B^* B}\). Thus (\ref{ineqharnt}) implies
\begin{equation}\label{ineqimtbstb}
I - T_B^* T_B \geq \frac{1}{4} T_{(1 - |z|^2) D_A^2}.
\end{equation}
Let $L^2_{\alpha, \HE}$ be the $\HE$-valued $L^2_\alpha$ space on $\D$. Let $P_\alpha$ be the orthogonal projection from $L^2_{\alpha, \HE}$ onto $A^2_{\alpha,\HE}$. We define the Hankel operator $H_{B^*}$ by
\begin{align*}
H_{B^*}f=(I-P_{\alpha})(B^{\ast}f), \quad f\in A_{\alpha,E}^{2}.
\end{align*}
Then one checks that
\begin{align}\label{toephanlident}
 T_{BB^*}-T_BT_B^*=H_{B^*}^*H_{B^*}\geq0.
\end{align}
Thus
\begin{align}\label{inebeimtbtbsdas}
 I-T_BT_B^*
 &=T_{I_\HE-BB^*}+H_{B^*}^*H_{B^*}\\
 &\geq T_{I_\HE-BB^*} \notag\\
 &\geq \frac14 T_{(1-|z|^2)D_{A^*}^2}.\notag
\end{align}

Let $G=D_{A^*}$ on $\mathcal{E}$ and $Q_G = T_{(1-|z|^2)G^2}$ on $A^2_{\alpha,\mathcal{E}}$.
We now determine $\mathcal{M}(Q_G^{1/2})$. Note that
\[
K_\alpha(z,w)=\frac{1}{(1-z\overline{w})^{\alpha+2}}
=\sum_{n=0}^{\infty} c_{n,\alpha}(z\overline{w})^n,\quad
c_{n,\alpha}=\frac{\Gamma(n+\alpha+2)}{n!\,\Gamma(\alpha+2)}.
\]
Furthermore,
\[
T_{1-|z|^2} z^n = \frac{\alpha+1}{n+\alpha+2}\, z^n.
\]
Therefore the reproducing kernel of $\mathcal{M}(Q_G^{1/2})$ is
\[
\begin{aligned}
L_G(z,w) &= T_{(1-|z|^2)G^2} K_\alpha(z,w)\\
&=\sum_{n=0}^{\infty} c_{n,\alpha}\,\frac{\alpha+1}{n+\alpha+2}\,(z\overline{w})^n G^2\\
&=\sum_{n=0}^{\infty} \frac{n+\alpha+1}{n+\alpha+2}\,c_{n,\alpha-1}\,(z\overline{w})^n G^2.
\end{aligned}
\]
Since
\[
\frac{\alpha+1}{\alpha+2}\le \frac{n+\alpha+1}{n+\alpha+2}<1,
\]
we obtain
\[
\frac{\alpha + 1}{\alpha + 2} \frac{G^2}{(1 - z\overline{w})^{\alpha+1}} \preceq L_G(z, w) \preceq \frac{G^2}{(1 - z\overline{w})^{\alpha+1}},
\]
where \(R(z, w) \preceq K(z, w)\) means \(K(z, w) - R(z, w)\) is positive semi-definite, i.e. a kernel.
Note that \(\mathcal{M}(G) = \mathcal{M}(D_{A^*}) = \mathcal{E}_1\), and the reproducing kernel \(\frac{G^2}{(1 - z\overline{w})^{\alpha+1}}\) determines the
space \(A^2_{\alpha-1, \mathcal{E}_1}\). It thus follows that
\[
\mathcal{M}(Q_G^{1/2}) \approx A^2_{\alpha-1, \mathcal{E}_1},
\]
see e.g. \cite{PaulsenRag}.
Thus by (\ref{inebeimtbtbsdas}) and Douglas range inclusion lemma (\cite{Do66}), we obtain
\[
A^2_{\alpha-1, \mathcal{E}_1} \approx \mathcal{M}(Q_G^{1/2}) \subseteq \mathcal{M}((I - T_B T_B^*)^{1/2}) = \mathcal{A}_\alpha(B).
\]
Similarly, by using (\ref{ineqimtbstb}), we obtain that
\[
A^2_{\alpha-1, \mathcal{E}_2} \approx \mathcal{M}((T_{(1 - |z|^2) D_A^2})^{1/2}) \subseteq \mathcal{M}((I - T_B^* T_B)^{1/2}) = \mathcal{A}_\alpha(B^*).
\]
The proof is complete.
\end{proof}

\subsection{Consequences of Theorem \ref{conjecproof}}
We conclude this section with short proofs of two results from \cite{GLM26}.

\begin{theorem}[\cite{GLM26}, Theorems 3.14 and 3.16]\label{mainGLM}
Let $B\in H^{\infty}_1(L(\HE)), \alpha > -1$. Then $\mathcal{A}_{\alpha}(B)\supseteq A_{\alpha-1,\HE}^{2}$ if and only if
$\left\Vert B(0)\right\Vert <1$, which is equivalent to $\mathcal{A}_{\alpha}(B^*)\supseteq A_{\alpha-1,\HE}^{2}$.
\end{theorem}

\begin{proof}
Lemma \ref{pureope} implies that if $\|B(0)\| < 1$, then $\HH(B(0)) = \HE = \HH(B(0)^*)$. Thus if $\|B(0)\| < 1$, then Theorem \ref{conjecproof} ensures that
$\mathcal{A}_{\alpha}(B)\supseteq A_{\alpha-1,\HE}^{2}$ and $\mathcal{A}_{\alpha}(B^*)\supseteq A_{\alpha-1,\HE}^{2}$.

Note that $\mathcal{A}_{\alpha}(B)\supseteq A_{\alpha-1,\HE}^{2}$ if and only if there is some constant $C > 0$ such that
\begin{align*}
F_\alpha^{B}(z,w) = \frac{I_\HE - B(z) B(w)^*}{(1-z\overline{w})^{\alpha+2}} \succeq C \frac{I_{\HE}}{(1-z\overline{w})^{\alpha+1}},
\end{align*}
where $F_\alpha^{B}(z,w) = \frac{I_\HE - B(z) B(w)^*}{(1-z\overline{w})^{\alpha+2}}$ is the reproducing kernel of $\mathcal{A}_{\alpha}(B)$ given by (\ref{reprokernelaabbs}). Thus if $\mathcal{A}_{\alpha}(B)\supseteq A_{\alpha-1,\HE}^{2}$, then by letting $z = w = 0$ in the above expression, we obtain that
\begin{align*}
I_\HE - B(0) B(0)^* \geq C I_\HE.
\end{align*}
It then follows that $\|B(0)\| < 1$.

Similarly, if $\mathcal{A}_{\alpha}(B^*)\supseteq A_{\alpha-1,\HE}^{2}$, then $\|B(0)\| < 1$. The proof is complete.
\end{proof}

\begin{prop}[\cite{GLM26}, Proposition 3.20]
\label{dense}
Let $B\in H^{\infty}_1(L(\HE)), \alpha > -1$. Then $\mathcal{A}_{\alpha}(B)$ and
$\mathcal{A}_{\alpha}(B^{\ast})$ are dense subspaces of
$A_{\alpha,\HE}^{2}$ if and only if $B$ is pure.
\end{prop}
\begin{proof}
By the decomposition (\ref{decomdefect}), we see that the necessity holds.

The sufficiency part follows from Theorem \ref{conjecproof} and Lemma \ref{pureope}.
\end{proof}

\section{An example in infinite dimension}
The following finite dimensional rigidity theorem was proved in \cite[Theorem 1.4]{GLM26}.

\begin{theorem}[\cite{GLM26}, Theorem 1.4]
	If $B\in H^{\infty}_1(L(\HE))$ and $\HE$ is finite dimensional, then
	$\mathcal{A}_{\alpha}(B)\approx\mathcal{A}_{\alpha}(B^{\ast})\approx A_{\alpha-1,\HE}^{2}$ for $\alpha>-1$ if and
	only if $B$ is a finite Blaschke-Potapov product with $\left\Vert
	B(0)\right\Vert <1.$
\end{theorem}

We show that the above result can not be extended to the case that $\HE$ is infinite dimensional. Theorem \ref{mainthm2} follows from the following result.
\begin{theorem}\label{anexampleinf}
There exist an
infinite dimensional Hilbert space $\HE$ and a function $B\in H^{\infty}_1(L(\HE))$
such that for every $\alpha>-1$,
\begin{align*}
\HA_\alpha(B) \approx \HA_\alpha(B^*) \approx A_{\alpha-1,\HE},
\end{align*}
while $B$ is not a finite Blaschke-Potapov product. In addition,
$\|B(0)\|<1$ and $B^{\ast}(z)B(z)=B(z)B^{\ast}(z) = I_{\HE}$ for a.e. $\left\vert
z\right\vert =1.$
\end{theorem}
\begin{proof}
Let $\HE = \ell^2(\mathbb{N})$ with standard orthonormal basis $(e_n)_{n\geq 1}$. Let
\begin{align*}
a_n = \frac{3}{4} - \frac{1}{n+1}, \quad n \geq 1,
\end{align*}
and define the diagonal operator \(A \in L(\mathcal{E})\) by \(Ae_n = a_n e_n, n \geq 1\). Then \(a_n > 0\) and
\[
\|A\| = \sup_{n \geq 1} a_n = \frac{3}{4} < 1.
\]
Define
\[
B(z) = (zI_{\mathcal{E}} - A)(I_{\mathcal{E}} - zA)^{-1}, \quad z \in \mathbb{D}.
\]
That is,
\[
B(z)e_n = \varphi_{a_n}(z)e_n, \quad \varphi_{a_n}(z) = \frac{z - a_n}{1 - a_n z}.
\]
Since \(\|A\| < 1\), we have \(B\) is analytic in the operator norm on a neighborhood of the
closed unit disk. Moreover,
\[
\|B(z)\| = \sup_{n \geq 1} |\varphi_{a_n}(z)| \leq 1, \quad z \in \mathbb{D}.
\]
Thus \(B \in H_1^\infty(L(\mathcal{E}))\). We also have
\[
B(0) = -A, \quad \|B(0)\| = \frac{3}{4} < 1.
\]
For every \(\zeta \in \mathbb{T}\) and every \(n \geq 1\), we have \(|\varphi_{a_n}(\zeta)| = 1\). So
\[
B(\zeta)^* B(\zeta) = B(\zeta) B(\zeta)^* = I_{\mathcal{E}}, \quad \zeta \in \mathbb{T}.
\]
Notice also that \(B(z)\) is normal for every \(z \in \mathbb{D}\).

We next identify \(\mathcal{A}_\alpha(B)\). First observe that \(B(w)^* = (I_{\mathcal{E}} - \overline{w}A)^{-1}(\overline{w}I_{\mathcal{E}} - A)\). So
\begin{align*}
& (I_{\mathcal{E}} - zA)(I_{\mathcal{E}} - B(z)B(w)^*)(I_{\mathcal{E}} - \overline{w}A) \\
&= (I_{\mathcal{E}} - zA)(I_{\mathcal{E}} - \overline{w}A) - (zI_{\mathcal{E}} - A)(\overline{w}I_{\mathcal{E}} - A) \\
&= I_{\mathcal{E}} + z\overline{w}A^2 - z\overline{w}I_{\mathcal{E}} - A^2 \\
&= (1 - z\overline{w})(I_{\mathcal{E}} - A^2) \\
&= (1 - z\overline{w})D_A^2,
\end{align*}
where $D_A = (I_{\mathcal{E}} - A^2)^{1/2}$. It then follows that
\[
I_{\mathcal{E}} - B(z)B(w)^* = (1 - z\overline{w})(I_{\mathcal{E}} - zA)^{-1}D_A^2(I_{\mathcal{E}} - \overline{w}A)^{-1}.
\]
Let $L(z) = D_A(I_{\mathcal{E}} - zA)^{-1}$, then
\[
I_{\mathcal{E}} - B(z)B(w)^* = (1 - z\overline{w})L(z)L(w)^*.
\]
Therefore, the reproducing kernel of $\mathcal{A}_\alpha(B)$ is
\[
\begin{aligned}
F_\alpha^B(z,w) &= \frac{I_{\mathcal{E}} - B(z)B(w)^*}{(1 - z\overline{w})^{\alpha+2}} \\
&= L(z)\frac{I_{\mathcal{E}}}{(1 - z\overline{w})^{\alpha+1}}L(w)^*.
\end{aligned}
\]
Note that both $L$ and its inverse $L(z)^{-1} = (I_{\mathcal{E}} - zA)D_A^{-1}$ are analytic in the operator norm on a neighborhood of the closed unit disk. Hence for every $\alpha > -1$, they are bounded multipliers of $A_{\alpha-1,\mathcal{E}}^2$. It then follows that
\[
\mathcal{A}_\alpha(B) = L A_{\alpha-1,\mathcal{E}}^2 \approx A_{\alpha-1,\mathcal{E}}^2.
\]
Since $B(z)$ is normal for every $z\in\D$, by (\ref{toephanlident}), we have
\begin{align*}
T_B^* T_B - T_B T_B^* &= T_{B^* B} - T_B T_B^* \\
&= T_{B B^*} - T_B T_B^* = H_{B^*}^* H_{B^*} \geq 0.
\end{align*}
So \(I - T_B T_B^* \geq I - T_B^* T_B\) and \(\mathcal{A}_\alpha(B^*) \subseteq \mathcal{A}_\alpha(B)\). Note that \(\|B(0)\| < 1\). Thus by Theorem \ref{mainGLM} and \(\mathcal{A}_\alpha(B) \approx A_{\alpha-1, \mathcal{E}}^2\), we obtain that
\[
A_{\alpha-1, \mathcal{E}}^2 \subseteq \mathcal{A}_\alpha(B^*) \subseteq \mathcal{A}_\alpha(B) \approx A_{\alpha-1, \mathcal{E}}^2.
\]
Thus
\[
\mathcal{A}_\alpha(B^*) \approx \mathcal{A}_\alpha(B) \approx A_{\alpha-1, \mathcal{E}}^2.
\]

It remains to prove that \(B\) is not a finite Blaschke-Potapov product. Suppose to the contrary that
\[
B(z) = U \prod_{i=1}^N Q(z, b_i, F_i, \mathcal{E})
\]
for some unitary \(U\), \(b_i \in \mathbb{D}\), and closed subspaces \(F_i \subseteq \mathcal{E}\). Each factor can be written as
\[
Q(z, b_i, F_i, \mathcal{E}) = \frac{(z - b_i)P_{F_i^\perp} + (1 - \overline{b_i}z)P_{F_i}}{1 - \overline{b_i}z}.
\]
It follows that every scalar matrix coefficient
\[
z \longmapsto \langle B(z)x, y \rangle, \quad x, y \in \mathcal{E},
\]
is a rational function whose poles belong to the finite set
\[
\left\{ \frac{1}{\overline{b_1}}, \cdots, \frac{1}{\overline{b_N}} \right\},
\]
with the terms corresponding to \(b_i = 0\) omitted.

On the other hand,
\[
 \langle B(z)e_n,e_n\rangle
 =\varphi_{a_n}(z)
 =\frac{z-a_n}{1-a_nz},
\]
which has a pole at $z=1/a_n$. Since the numbers $a_n$ are distinct, these matrix coefficients have infinitely many
distinct poles. This is a contradiction. Therefore $B$ is not a finite Blaschke-Potapov product.
\end{proof}

\

\noindent \textbf{Acknowledgement}:
S. Luo was supported by National Science Foundation of China (No.12271149), Natural Science Foundation of Hunan Province (No.2024JJ2008).

\end{document}